\documentclass[11pt]{article} 
\usepackage{amssymb}
\usepackage{amsthm,amsmath,amssymb}
\usepackage{amsthm,amsmath}

\usepackage{amscd, amssymb}
\usepackage{hyperref}
\usepackage{comment}
\usepackage{enumitem}

\newtheorem{theorem}{Theorem}[section]
\newtheorem{lemma}[theorem]{Lemma}
\newtheorem{prop}[theorem]{Proposition}
\newtheorem{corollary}[theorem]{Corollary}
\newtheorem{conjecture}[theorem]{{Conjecture}}

\newtheorem{claim}[theorem]{{Claim}}

\theoremstyle{remark}
\newtheorem{remark}[theorem]{Remark}

\theoremstyle{definition}

\newtheorem{definition}[theorem]{{Definition}}

\def\bclaim{\begin{claim}}
\def\eclaim{\end{claim}}
\def\bdefin{\begin{definition}}
\def\edefin{\end{definition}}
\def\bcor{\begin{corollary}}
\def\ecor{\end{corollary}}
\def\bthm{\begin{theorem}}
\def\ethm{\end{theorem}}
\def\bconj{\begin{conjecture}}
\def\econj{\end{conjecture}}
\def\blem{\begin{lemma}}
\def\elem{\end{lemma}}
\def\blemma{\begin{lemma}}
\def\elemma{\end{lemma}}
\def\bprop{\begin{prop}}
\def\eprop{\end{prop}}
\def\bremark{\begin{remark}}
\def\eremark{\end{remark}}

\newcommand{\R}{\mathbb R}
\newcommand{\RR}{\mathbb R}
\newcommand{\C}{\mathbb C}
\newcommand{\CC}{\mathbb C}
\newcommand{\NN}{\mathbb N}

\newcommand{\ZZ}{\mathbb Z}

\renewcommand{\P}{\vec P}

\DeclareMathOperator{\spec}{spec}

\newcommand{\calF}{\mathcal F}

\newcommand{\la}{\left\langle}
\newcommand{\ra}{\right\rangle}

\def\USC{\operatorname{USC}}
\def\usc{\operatorname{usc}}
\def\diag{\operatorname{diag}}

\def\Im{\operatorname{Im}\,}

\def\Re{\operatorname{Re}\,}
\def\Sym{\operatorname{Sym}^2}
\def\Int{\hbox{\rm int}\,}

\def\Im{{\operatorname{Im}\,}}
\def\Re{{\operatorname{Re}\,}}

\def\tr{\hbox{\rm tr}}

\def\q{\quad}

\def\bpf{\begin{proof}}
\def\epf{\end{proof}}
\def\beq{\begin{equation}}
\def\eeq{\end{equation}}
\def\beqno{\begin{equation*}}
\def\eeqno{\end{equation*}}
\def\eaeq{\end{aligned}}
\def\baeq{\begin{aligned}}

\def\P{{\mathcal P}}

\def\SA{\hbox{\rm SA}}
\def\tF{\widetilde F}
\def\del{\partial}
\def\sm{\setminus}

\def\HL{Harvey--Lawson }

\def\n{\nabla}
\def\h#1{\hbox{#1}}

\def\lb{\label}

\def\opcit{\underbar{\phantom{aaaaa}}}
\def\th{\theta}

\def\i{\sqrt{-1}}
\def\del{\partial}

\def\ra{\rightarrow}
\def\eps{\epsilon}
\def\del{\partial}

\def\sm{\setminus}

\def\sm{\setminus}
\def\w{\wedge}

\def\beq{\begin{equation}}
\def\eeq{\end{equation}}
\def\bi#1{\bibitem{#1}}
\def\PSH{\mathrm{PSH}}

\def\h#1{\hbox{#1}}

\def\calS{{\mathcal S}}

\def\la{\lambda}
\def\Th{\Theta}

\def\Symn{\Sym(\RR^n)}

\def\wtTh{\widetilde{\Theta}}
\def\wtth{\widetilde{\theta}}

\begin{document}
\title{A minimum principle for Lagrangian graphs}
\author{Tam\'as Darvas and Yanir A. Rubinstein}
\date{}
\maketitle

\begin{abstract}
The classical minimum principle is foundational in convex and complex analysis
and plays an important r\^ole in the study of the real and complex Monge--Amp\`ere equations.
This note establishes a minimum principle in Lagrangian geometry.
This principle relates the classical Lagrangian angle of Harvey--Lawson
and the space-time Lagrangian angle introduced recently by Rubinstein--Solomon.
As an application, this gives a new formula for 
solutions of the 
degenerate special Lagrangian equation in space-time
in terms of the (time) partial Legendre transform of a family 
of solutions of obstacle problems for the (space)  non-degenerate 
special Lagrangian equation.

\end{abstract}



\section{Introduction}
\label{}

Suppose that $f$ is a convex function on $\RR\times\RR^n$. Then
$$
g(x):=
\inf_{s\in\RR}f(s,x)
$$ 
is either identically $-\infty$, or else a convex function on $\RR^n$
\cite[Theorem 5.7]{Rockafellar},\cite[Theorem 1.3.1]{Kiselman-notes}.
This is often referred to as the ``minimum principle" for convex functions.
If we replace ``convex" with ``plurisubharmonic" and $\RR$ by $\CC$ this is not true in general.
An important situation in which this is true was described by Kiselman in the 70's,
and we now state the simplest version of his theorem.
Let
$I\subset\RR$ be an open interval  and denote by
$$
S : = I + \i \RR \subset \Bbb C
$$ the strip associated to $I$.
Denote by $s$ the coordinate on $I$ and by $\tau :=s+\i t$ the complex
coordinate on $S$.

\begin{theorem}\label{KiselmanThm} {\rm (Kiselman's principle \cite{Kiselman})} Let $D \subset \Bbb C^n$ be a domain.
If $v\in\PSH(S\times D)$ is such that $v(s+\i t,z)=v(s,z)$ for all $s \in I$, then
\begin{equation}\label{Kiselman}
v(z)= \inf_{\tau \in S}v(\tau,z) 
\end{equation}
is either identically $-\infty$, or else plurisubharmonic on $D$.
\end{theorem}

Commonly, the supremum of a family of subsolutions of an equation (in this case, plurisubharmonic functions
are subsolutions for the  homogeneous complex Monge--Amp\`ere equation) is again a subsolution. The unexpected feature of this result is that the same can be said about an infimum. As one might expect, this has important implications to certain partial differential equations (PDE) and (pluri)potential theory \cite{Kiselman,Kiselman-notes}.

One natural way to generalize Kiselman's principle would be to consider classes of functions
other then convex or plurisubharmonic functions. A natural candidate is given by the notion of a subequation introduced by Harvey--Lawson, and in a different guise by Slodkowski
\cite{HL,Slodkowski,Nirenberg}.
A subequation is, roughly, a class of functions that serve as subsolutions for an elliptic PDE of second order.
However, it turns out that a simpleminded generalization is false for general subequations.

What we achieve in this note is 
a minimum principle for a particular subequation of interest in Lagrangian geometry. 
The result is stated in Theorem \ref{mainThm}.
The interest in this subequation comes from the associated PDE. 
In the case of convex/plurisubharmonic functions the associated equation is the homogeneous real/complex Monge--Amp\`ere equation, and the minimum/Kiselman principle has important implications to the study of its solutions, as shown recently by 
Ross--Witt Nystr\"om and the 
authors \cite{RWN,DR}. In the case studied in this article, the associated equation is the degenerate special Lagrangian equation (DSL) introduced recently by Rubinstein--Solomon \cite{RS}.
Inspired by the main result of \cite{DR}, we show how the minimum principle established in this
article can be applied to the study of the DSL. In particular, in Theorem \ref{geodThm} we derive a new formula for the 
weak solutions of the DSL constructed in  \cite[Theorem 1.2]{RS}. 

Our results can be viewed in the framework of a program initiated in \cite{RS} to develop a potential theory for the (degenerate) special Lagrangian equation and weak geodesics in the space of positive
Lagrangians with a view towards the strong Arnold conjecture \cite[\S2]{RS}
and as part of a program initiated by Solomon \cite{S1,S2} (see also \cite{SY}) to understand the existence and uniqueness of special Lagrangian submanifolds in Calabi--Yau manifolds.

This note is organized as follows.
In Section \ref{sec: preliminaries} we recall the subequations corresponding to the special Lagrangian equation and to the 
degenerate special Lagrangian equation \cite{HL,RS}. 
Section \ref{sec: MainThm} is devoted to the proof of the minimum principle for Lagrangian graphs
(Theorem \ref{mainThm}). In 
Section \ref{sec: applications} we give the proof of Theorem \ref{geodThm} concerning solutions of the DSL.

\section{The special Lagrangian subequation and the degenerate special Lagrangian subequation}
\label{sec: preliminaries}

This section recalls basic notions from \cite{HL,RS}.

\subsection{Subequations}
\label{}
A subequation is a proper closed subset
$F$ of the set of $m$-by-$m$ symmetric matrices 
that is invariant under translation by positive matrices. 
A subequation $F$ is said to be associated to a PDE
of the form 
\begin{equation}\label{eq: Hessianeq}
f(\nabla^2u(x))=0, 
\q x \in U\subset\RR^m,
\end{equation}
if $C^2(U)$ solutions
 of the equation satisfy $\nabla^2u(x)\in\del F$
for each $x\in U$. A subequation $F$ gives rise to a natural notion of subsolutions,
also called functions of type $F$, denoted $F(U)$. 
Namely, $u \in C^2(U)$ is a subsolution,
denoted $u\in F(U)$, if $\nabla^2u(x) \in F$ for all $x \in U.$ However,
elements of $F(U)$ are typically only upper semicontinuous and are defined using a viscosity type condition, as we detail below. 
These functions are 
the key object in a so-called potential theory associated to the PDE \eqref{eq: Hessianeq},
in a similar way to, e.g., subharmonic functions and the Laplace equation, or plurisubharmonic functions and the homogeneous complex Monge--Amp\`ere equation.

A subequation $F$ gives rise to a weak version of the Dirichlet problem for each domain $U \subset \RR^m.$ \HL show existence and uniqueness of continuous solutions to the $F$-Dirichlet problem under certain assumptions on the boundary of $U.$ Making connection with the classical theory, if the continuous solution is in $C^2(U)$, it must be a solution in the classical sense. 

Let us recall in more detail the 
basic notions and notation concerning subequations, following Harvey--Lawson  \cite{HL}.
Denote by $\Sym(\R^m)$ the set of all symmetric 
$m$-by-$m$ matrices, and by $\P$ the subset of nonnegative matrices.
A proper nonempty closed subset $F$ of $\Sym(\R^m)$ is a {\it subequation} 
if  \cite[Definition 3.1]{HL}
\beq
\lb{DirichletSetPropEq}
F+\P\subset F.
\eeq
Denote by $\Int S$ the interior of a set $S$, and 
by $S^c$ its complement.
By $\widetilde F$ we denote the {\it dual set to $F$},  which is also a subequation, and is defined by
$$
\tF:=(-\Int F)^c.
$$
A function $u \in \USC(D)$ is {\it subaffine}, denoted $u\in\SA(D)$, if for all affine functions $a$ and $K\subset D$ compact, $u\le a$ on $\del K$ implies $u\le a $ on $K$. \HL prove that
such functions satisfy the maximum principle~\cite[Proposition 2.3]{HL},
\begin{equation}\label{eq:mp}
\h{\rm if \ \ } u \in \SA(D) \h{\ \ \  then \ \ } \sup_{\overline{D}} u =\sup_{\partial D} u.
\end{equation}
A function $u\in \USC(D)$ is of {\it type $F$}, denoted $u\in F(D)$, if $u+v\in\SA(D)$
for all $v\in C^2(D)$ satisfying $\n^2 v(x)\in \tF$, for all $x\in D$. 

From now on, unless stated otherwise, we assume that $D$ is bounded. 
The elements of $F(D)$ serve as subsolutions to the PDE associated to $F$.
Similarly to subharmonic functions they satisfy many useful properties, 
due to Harvey--Lawson \cite{HL},  that we will use repeatedly. The reader may find the useful list of most of the properties we will make use of in  \cite[\S6]{RS}.

To give classical examples, the subequation whose associated subsolutions are convex functions is 
the set of nonnegative matrices, while plurisubharmonic functions are associated to
the subequation of nonnegative Hermitian matrices.

\subsection{The special Lagrangian subequation}
\label{}

A family of subequations associated to 
all branches of the special Lagrangian equation was introduced by Harvey--Lawson, 
\begin{equation*}
F_c:=\{A\in\Symn\,:\,\tr\tan^{-1}A\ge c\}.
\end{equation*}
Here, $-{n}\pi/2 < c < {n}\pi/2$ and the dual subequation is $\tilde F_c=F_{-c}$
\cite[Proposition 10.4]{HL}.

There is a relation between the subequation $F_c$ and the Lagrangian angle
of a Lagrangian graph. Indeed,  
the restriction of the form $dz^1~\w~\cdots~\w~dz^n$ to the Lagrangian graph 
$\{(x,\nabla u(x))\,:\, x\in\RR^n\}$ is equal to the volume form induced on the graph from the Euclidean metric on $\RR^{2n}$, up to a unit complex number (that depends on $x$)
\cite[Proposition 1.14, p. 89]{HL82}.
The argument of that number, denoted $\theta_u(x)\in S^1$ is called the Lagrangian angle at $(x,\nabla u(x))$, and a computation shows that
\begin{equation}
\begin{aligned}
\label{}
\th_u(x)= \arg \det(I + \i  \nabla^2u(x)).
\end{aligned}
\end{equation}
Here, we consider $S^1$ as an abelian group and use additive notation for the group law and the inverse.
Also,  we let $\arg$ denote the branch of the argument function with image in $(-\pi,\pi]$
(i.e., the branch whose domain is the complex plane minus the nonnegative real axis).
Then $\tan^{-1}\la:=\arg(1+\i\la)$,
for $\la\in \R$, where $\tan^{-1}$ denotes the branch of 
the inverse to $\tan$ with image in $(-\pi/2,\pi/2)$.

Motivated by this, one defines the Lagrangian angle of the symmetric matrix $A$ 
\[
\theta: \Symn \to S^1,
\]
 by
\[
\theta(A):= \arg \det(I + \i A).
\]
For $B \in \Sym(\CC^m),$ denote by $
\spec(B)\subset \CC$ 
the set of its eigenvalues, and for $\lambda \in \spec(B),$ denote by $
m(\lambda)$ 
the multiplicity of $\lambda$. Then, 
$$
\arg B=\sum_{\la\in\spec(B)}m(\la)\arg\la.
$$
One defines the lifted Lagrangian angle
\begin{equation}
\label{eq:wtth}
\wtth : \Symn \to \RR,  \qquad
\wtth(A):= 
\tr\arg(I+\i A).
\end{equation}
The name is justified by the fact that $\th\equiv \tilde\th \,\mod 2\pi$.
Observe that   \eqref{eq:wtth} makes sense since the eigenvalues of $I+\i A$ all have real part equal to one,
so these eigenvalues are all in the domain of $\arg$. Thus,
$$
F_c=\{A\in\Symn\,:\, \wtth(A)\ge c\}.
$$ 

The relation between the subequation $F_c$ and the special Lagrangian potential equation is as follows.
First, a function $v\in C^2(D)$ is said to solve 
the special Lagrangian potential equation of phase $c$ if its associated 
(lifted) Lagrangian angle is constant and equal to $c$, i.e.,
\begin{equation}
\begin{aligned}
\label{SL1Eq}
\wtth_v(x)=\tr\arg(I+\i\nabla^2 v(x))
&=c, \q \forall x\in D.
\end{aligned}
\end{equation}
From the definitions it then follows that 
a function $v\in C^2(D)$ satisfies  \eqref{SL1Eq} if and only if 
$ v\in F_c(D)\cap(-F_{-c}(D))\cap C^2(D)$.
Motivated by this, a function $v$ is said to be a $C^0$ weak solution of the special Lagrangian equation 
if 
$ 
v\in F_c(D)\cap(-F_{-c}(D))\cap C^0(D)$.

For a short summary of some of the key potential theoretic results of Harvey--Lawson \cite{HL} concerning $F_c(D)$, we refer to \cite[Section 6.4]{RS}

\subsection{The degenerate special Lagrangian subequation}
\label{DLSSubSec}

In recalling the constructions of \cite{RS}, we set the following notation.
For 
$$
C=[c_{ij}]_{i,j=1}^{n+1}\in \textup{Sym}(\Bbb C^{n+1}),
$$ we will  make frequent use of the block decomposition
\begin{equation*}
\begin{aligned}
\label{}
C=
\begin{pmatrix}
c_{00}&\vec c_0\cr
\vec c_0^T& C^+\cr
\end{pmatrix},
\end{aligned}
\end{equation*}
where $c_{00} \in \Bbb C, \vec c_0 \in \Bbb C^n$ and 
$C^+ \in \textup{Sym}^2({\Bbb C}^n)$. 
For $\eta\ge0$, write
\begin{equation}
\begin{aligned}
\label{IetaEq}
I^\eta_{n}:=\diag(\eta,1,\ldots,1)\in\Sym(\R^{n+1}).
\end{aligned}
\end{equation}
We also denote 
\begin{equation}
\begin{aligned}
\label{}
I_{n}:=I_{n}^0=\diag(0,1,\ldots,1)\in\Sym(\R^{n+1}),
\end{aligned}
\end{equation}
\beq
\label{calSEq}
\calS=\{A\in \textup{Sym}(\Bbb R^{n+1}): \ \ \det(I_{n}+\i A)=0 \}.
\eeq
It follows from \cite[Lemma 3.4]{RS} that in fact
$$
\calS=\{A\in \textup{Sym}^2(\Bbb R^n): \ A = \textup{diag}(0,A^+) \}.
$$ The reason why this set has special significance comes from the fact that the space-time Lagrangian angle $
\Theta: \textup{Sym}^2(\Bbb R^{n+1}) \sm\calS \to S^1$ defined by
\[
 \Theta(A) := \arg\det(I_{n} + \i A),
\]
does not extend continuously to $\cal S$, though it is smooth on $\textup{Sym}^2(\Bbb R^{n+1}) \setminus \cal S$. Fortunately, when considering $\wtTh$, the lift of the space-time Lagrangian angle to $\Bbb R$, it is possible to find a well-behaved upper semicontinuous extension to $\textup{Sym}^2(\Bbb R^{n+1})$, given by
\begin{equation}\label{eq:whTh}
\wtTh(A) := 
\begin{cases}
\sum_{\lambda \in \spec(I_{n} + \i A)} m(\lambda)\arg(\lambda),& A\in \textup{Sym}^2(\Bbb R^{n+1}) \sm \calS,\\
\pi/2 + \sum_{0\neq\lambda \in \spec(I_{n} + \i A)} m(\lambda)\arg(\lambda),&  A\in\calS,
\end{cases}
\end{equation}
where $m(\lambda)$ is the multiplicity of the eigenvalue $\lambda$.
More precisely,  the following is known
\cite[Theorem 3.1]{RS}.

\begin{theorem}
\lb{USCThm}
The function $\wtTh$  is the smallest upper semicontinuous function on $\Sym(\Bbb R^{n+1})$ 
extending $\wtTh|_{\Sym(\Bbb R^{n+1})\sm \calS}.$
\end{theorem}

For $c\in(-(n+1)\pi/2,(n+1)\pi/2)$, define $\mathcal F_{c}$ by
\beq
\label{calFcEq}
\calF_{c}
:=\big\{A \in \Sym(\R^{n+1})\,:\, 
\wtTh(A)\geq c
\big\}.
\eeq
From the semicontinuity of $\wtTh$ it follows that $\mathcal F_c$ is closed. The property $\mathcal F_c + \mathcal P \subset \mathcal F_c$ is 
proved in \cite[Lemma 5.3]{RS}, yielding that 
$\mathcal F_c$ is a subequation. Additionally, 
the dual subequation satisfies 
$\widetilde {\mathcal F}_c = \mathcal F_{-c}$ \cite[Lemma 5.5]{RS}.

The relation between the subequation $\mathcal F_c$ and the degenerate special Lagrangian potential equation is as follows.
Given an open set $D \subset \Bbb R^n$, an open interval $I \subset \Bbb R$, and a function $v\in C^2((0,1)\times D)$, $v$ is said to solve 
the degenerate special Lagrangian potential equation (DSL) of phase $c$ if its associated  (lifted) space-time Lagrangian angle is constant and equal to $c$, i.e.,
\begin{equation}
\begin{aligned}
\label{DSL1Eq}
\wtTh_v(t,x)=\tr\arg(I_n+\i\nabla^2 v(t,x))
&=c, \q \forall (t,x)\in I \times D.
\end{aligned}
\end{equation}
From the definitions it then follows that 
a function $v\in C^2(I \times D)$ satisfies  \eqref{DSL1Eq} if and only if 
$ v\in \mathcal F_c(I \times D)\cap(-\mathcal F_{-c}(I \times D))\cap C^2(I \times D)$.
Motivated by this, a function $v$ is said to be a $C^0$ weak solution of the degenerate special Lagrangian equation 
if 
$$ 
v\in \mathcal F_c(I \times D)\cap(-\mathcal F_{-c}(I \times D))\cap C^0(I \times D).
$$

\section{The minimum principle}
\label{sec: MainThm}

The main result of this section is the following minimum principle for Lagrangian graphs.

\begin{theorem}
\label{mainThm}
Let $D \subset \RR^n$ be a bounded domain and let $I\subset \RR$ be an open interval. Let $u\in\calF_{c}(I\times D)$ with $c \in [n\pi/2,(n+1)\pi/2)$. 
Then
\begin{equation}\label{Kiselman}
v(x):= \inf_{t \in I}u(t,x)
\end{equation}
belongs to $F_{c-\pi/2}(D)$.
\end{theorem}

\begin{remark}
By definition, the function $-\infty$ belongs to $F(D)$ for any subequation $F$. This explains
the apparent difference in the conclusion of Theorem \ref{mainThm} from the classical minimum principle.
\end{remark}

The proof of Theorem \ref{mainThm} will occupy the rest of the present section.

We start with the following observation. It is essentially contained in   \cite[Lemmas 3.6--3.7]{RS}.

\begin{lemma}
\label{AAplusLemma}
	For all $A \in \Sym(\Bbb R^{n+1})$
\begin{equation}
\begin{aligned}
\label{wtThwtthEq}
\wtTh(A) - \wtth(A^{+}) = \arg\big(
\i a_{00} + \vec a_0(I + \i A^+)^{-1}\vec a_0^T)\in [-\pi/2,\pi/2].
\end{aligned}
\end{equation}
\end{lemma}

We also need the following elementary fact.

\blem
\lb{HLPosLemma}
Let $C\in\Sym(\RR^m)$.  Then:
\hfill\break
(i)
$\Re\big(( I+\i C)^{-1}\big)$ is positive definite.
\hfill\break
(ii) $C$ is positive (semi)-definite if and only if $\Im\big(( I+\i C)^{-1}\big)$ is 
negative (semi)-definite.
\elem

\begin{proof}
For (i) see \cite[p. 94]{HL82}.
If $O\in O(m)$ diagonalizes $C$, i.e.,
$C=O^T\diag(\la_1(C),\ldots,\la_m(C))O$,
then
\beq\label{RePartInvEq}
\baeq
\Im\big(( I+\i C)^{-1}\big)
&=
O^{-1}\diag\bigg(
\frac{-\la_1(C)}{1 +\lambda^2_1(C)},\ldots,
\frac{-\la_m(C)}{1 +\lambda^2_m(C)}\bigg)O^{-T},
\eaeq
\eeq
proving (ii).
\end{proof}

\begin{lemma}
\label{a00Lemma}
Let $A \in \calF_{c}$ with $c \in [n\pi/2,(n+1)\pi/2)$. Then $a_{00}\ge 0$.

\end{lemma}

\begin{proof}
If $A\in\calS$ we are done since then $a_{00}=0$. Suppose $A\not\in\calS$.
If $a_{00}<0$, then 
Lemma~\ref{HLPosLemma} gives that
\[
\Im\left(\i a_{00} + \left\langle \vec a_0,(I + \i A^+)^{-1}\vec a_0\right\rangle\right) 
= a_{00}+ \left\langle \vec a_0, \Im(I + \i A^+)^{-1}\vec a_0\right\rangle < 0,
\]
which combined with \eqref{wtThwtthEq} implies that
\begin{equation}
\begin{aligned}
\label{wtThwtth2Eq}
\wtTh(A) - \wtth(A^+) \in [-\pi/2,0),
\end{aligned}
\end{equation}
a contradiction with the fact that $\wtTh(A)\ge c\geq n\pi/2$ since $\wtth(A^+)<n\pi/2$.
Thus, $a_{00}\ge 0$. 
\end{proof}

Combining Lemma \ref{a00Lemma} with results of Harvey--Lawson
\cite{HL} gives the following partial convexity statement. This is reminiscent of the hypothesis in Kiselman's theorem but holds in our setting without further assumption, as in the setting of convex functions.

\begin{lemma}
\label{cvxLemma}
Let $c \in [n\pi/2,(n+1)\pi/2)$ and $u \in \calF_{c}(I\times D)$.
For all $x\in D$, the function $t \to u(t,x)$ is
convex on $I$. 
\end{lemma}

\begin{proof} Let $I' \subset I$ and $D' \subset D$ be arbitrary relatively precompact open sets. From \cite[Theorem 8.2]{HL} it follows that we can find $\{ u_k\}_k \in \mathcal F_c(I' \times D')$ quasi--convex, such that $u_k \searrow u|_{I' \times D'}$. We know that $u_k$ is twice differentiable and $\nabla^2 u_k \in \mathcal F_c$ for a.e. $(t,x) \in I' \times D'$ . Using Fubini's theorem and Lemma \ref{a00Lemma}, we obtain that,  for a.e. $x \in D'$, the function $t \to u_k(t,x)$ is twice differentiable and $\nabla^2_t u_k(t,x) \geq 0$  for a.e. $t \in I'$.

Since $t \to u_k(t,x)$ is additionally quasi--convex, it follows that $t \to u_k(t,x)$ has to be convex on $I'$ for a.e. $x \in D'$ \cite[Corollary 7.5]{HL}. As each $u_k$ is continuous on $I' \times D'$, it follows that in fact $t \to u_k(t,x)$ has to be convex for all $x \in D'$. Letting $k \to \infty$ we obtain that $t \to u(t,x)$ is also convex for all $x \in D'$, finishing the proof.
\end{proof}

\begin{remark} {\rm
\label{}
Alternatively, Lemma \ref{cvxLemma} also follows from 
a more recent general restriction theorem of Harvey--Lawson~\cite{HL3}.

} \end{remark}

Let $t\in\RR, \, x\in\RR^n$.
For a function $f(t,x)$ of $n+1$ variables denote
\begin{equation}
\begin{aligned}
\label{}
 \nabla^2 f=
\begin{pmatrix}
 \ddot f& \nabla^2_{tx}f\cr
 (\nabla^2_{tx}f)^T& \nabla^2_x f\cr
\end{pmatrix}.
\end{aligned}
\end{equation}

The next lemma is modeled on Kiselman's proof of the classical minimum principle  \cite[Theorem 1.3.1]{Kiselman-notes}. 

\begin{lemma}
\label{MinLemma}
Suppose that $f\in C^2(I\times D)$, and that for each $x\in D$,
$f(\,\cdot\,,x):I\ra \RR$ is strongly convex 
and achieves its unique infimum at the point $t(x)$ in the interior of $I$. 
Denote by
$$
g(x)=\inf_{t\in D}f(t,x)=f(t(x),x), \q x\in D,
$$
Then,
\begin{equation}
\begin{aligned}
\label{HessgEq}
 \nabla^2 g(x)= \bigg(\nabla^2_x f-\frac{1}{\ddot f}(\nabla^2_{tx}f)^T \nabla^2_{tx}f\bigg)(t(x),x).
\end{aligned}
\end{equation}
\end{lemma}

\begin{proof}First we claim that $g\in C^2$.
To see this, let $t=t(x)$ be the unique
solution of 
\begin{equation}
\begin{aligned}
\label{vEq}
g(x)=f(t(x),x).
\end{aligned}
\end{equation}
Since $f\in C^1$, $t(x)$ is the unique solution of
\begin{equation}
\begin{aligned}
\label{nablatvEq}
\dot f(t(x),x)=0.
\end{aligned}
\end{equation}
By the implicit function theorem, $t(x)$ is a $C^1$ function of $x$ provided 
$\ddot f>0$,
which holds by assumption.
Thus, $g\in C^1$ by \eqref{vEq}. Differentiating \eqref{vEq} and evaluating at $(t(x),x))$ then
gives 
\begin{equation}
\begin{aligned}
\label{nablaxvEq}
\nabla_x g(x)=\nabla_x f(t(x),x),
\end{aligned}
\end{equation}
using   \eqref{nablatvEq}. Since the right-hand side is differentiable it follows that $g\in C^2$, as claimed; moreover,
\begin{equation}\nonumber
\begin{aligned}
\label{hessxvEq}
\nabla^2_x g(x)=\nabla^2_x f(t(x),x)+\nabla^2_{tx}f(t(x),x)^T\nabla_x t(x).
\end{aligned}
\end{equation}
Now, using   \eqref{nablatvEq},
\begin{equation}\nonumber
\begin{aligned}
\label{}
\ddot f(t(x),x)\nabla_x t(x)+\nabla^2_{tx}f(t(x),x)=0,
\end{aligned}
\end{equation}
i.e.,
\begin{equation}\nonumber
\begin{aligned}
\label{}
\nabla_x t(x)=-\big(\ddot f(t(x),x)\big)^{-1}\nabla_{tx}^2f(t(x),x).
\end{aligned}
\end{equation}
Thus,
\begin{equation}\nonumber
\begin{aligned}
\label{hessxv2Eq}
\nabla^2_x g(x)
=
\Big(
\nabla^2_x f-\frac{1}{\ddot f}(\nabla_{tx}^2 f)^T
\nabla^2_{tx}f\Big)(t(x),x),
\end{aligned}
\end{equation}
as claimed.
\end{proof}

\begin{corollary}
\label{detHessCor}Under the assumptions of Lemma \ref{MinLemma},
\begin{equation}
\begin{aligned}
\label{hessxv2Eq}
\det\big(I+\i\nabla^2_x g(x)\big)
=
\frac{\det\big(I_{n}+\i\nabla^2 f(t(x),x)\big)}
{\i\ddot f(t(x),x)} 
.
\end{aligned}
\end{equation}
In particular $
\wtth_g(x)=\wtTh_f(t(x),x)-\pi/2$.
\end{corollary}
\begin{proof}
Note that for $C=[c_{ij}]_{i,j=0}^n\in \Sym(\C^{n+1})$ with
$c_{00}\not=0$,
\beq
\label{DetFormulaEq}
\det C
= 
c_{00}
\det
\big(
C^+ - (\vec c_0)^T\vec c_0/{c_{00}}\big).
\eeq
(Note that this identity is different from the one used to prove Lemma \ref{AAplusLemma}!)
Applying this to $C=I_{n}+\i\nabla^2f$, combined with Lemma \ref{MinLemma} above  gives \eqref{hessxv2Eq}.

We turn to the last statement.   
As $\ddot f>0$, equation \eqref{hessxv2Eq} gives that
$$
\pi/2 =\Th_f(t(x),x)-\th_g(x).
$$
So for some $p\in\ZZ$, $\pi/2-2\pi p =\wtTh_f(t(x),x)-\wtth_g(x)$. Lemma \ref{AAplusLemma} implies that the right-hand side is in $[-\pi/2,\pi/2]$. 
Thus, $p=0$, as desired.
\end{proof}

\begin{proof}[Proof of Theorem \ref{mainThm}]
Suppose $I_k \subset I$ and $D_k \subset D$, $k \in \Bbb N$ are exhaustions 
of $I$ and $D$, respectively, by precompact open subsets.

First, observe that it is enough to prove that $v_k \in F_{c-\frac{\pi}{2}}(D_k)$, where 
$$
v_k(x):= \inf_{t \in I_k}u(t,x),\q x \in D_k.
$$
Indeed, the sequence $\{ v_k\}_k$ is decreasing, hence the limit $v:= \lim_k v_k$ 
satisfies $v \in F_{c-\frac{\pi}{2}}(D)$ \cite[(5), p. 410]{HL}.

Let us fix $k$. By Lemma \ref{regularizeLemma} below, there exists a decreasing sequence $u_k^l \in \mathcal F_c\cap C^\infty(I_k \times D_k) $ such that 
$u_k^l \searrow u|_{I_k \times D_k}$. Pick $f: I_k \to \Bbb R$, a strongly convex smooth exhaustion function of $I_k$ (e.g., if $I_k=(a,b)$, take $f=-\log(t-a)-\log(b-t)$). After adding $\frac{1}{l}f$ to $u_k^l$, we can further assume that for any $x \in D_k$, the function $u_k^l(\cdot,x)$ is a strongly convex exhaustion of $I_k$. 

Observe that $u^l_k \in C^2(I_k \times D_k)$ satisfies the requirements of 
Lemma \ref{MinLemma}, hence by the last statement in Corollary  \ref{detHessCor}, we obtain that  
$$\tilde \theta_{v^l_k}(x)=\wtTh_{u^l_k}(t^l_k(x),x)-\pi/2 \geq c - \pi/2, 
\q x \in D_k,
$$ where 
$$
v_k^l(x):= \inf_{t \in I_k}u_k^l(t,x), \q x \in D_k.
$$
Hence, by definition,  $v^l_k \in F_{c-\frac{\pi}{2}}(D_k).$ As the sequence $\{ v_k^l\}_l$ is decreasing, we ultimately get $v_k:=\lim_l v_k^l \in F_{c-\frac{\pi}{2}}(D_k)$, finishing the proof.
\end{proof}

\begin{lemma}
\label{regularizeLemma}
For any  $c \in [n\pi/2,(n+1)\pi/2)$ the set $\mathcal F_c \subset \Sym(\Bbb R^{n+1})$ is convex. If $u \in \calF_{c}(I\times D)$ and $I'\subset I, \ D'\subset D$ are  precompact open sets then there exists a sequence $\{u_k\}_k\subset 
C^\infty\cap \calF_{c}(I'\times D')$
strictly decreasing to $u|_{I' \times D'}$.
\end{lemma}

\begin{remark} Observe that we can not ask for uniform convergence of $u_k$ to $u$, as in \cite[Lemma 10.7]{RS}, because $u$ may not be continuous.
\end{remark}

\begin{proof} 
The first part of the proof is devoted to showing that $\mathcal F_{c}$ is convex.
Let $A,B \in \mathcal F_{c}$. We claim that there exists $A_k,B_k \in \mathcal F_{c} \setminus \mathcal S$ such that $A_k \to A, B_k \to B$ and 
$$
C_k := (A_k + B_k)/2 \not\in \mathcal S.
$$
This follows from the fact that $\mathcal F_{c}+\textup{int }\mathcal P \subset \textup{int }\mathcal F_{c}$, hence one has a great degree of freedom in perturbing $A,B$. In fact, if we perturb using elements of $\textup{int }\mathcal P$, we get additionally that 
$$
A_k,B_k \in \mathcal F_{c+ \varepsilon_k}
$$ for some $\varepsilon_k >0$ (see the first formula in the proof of \cite[Lemma 5.5]{RS}).

Recall from \eqref{IetaEq} that $I_{n}^p := \diag(p,1,1,...1)$. For $E \in \Sym(\Bbb R^{n+1})$, set 
$$
E_p := I_n^p E I_n^p.
$$
As in the proof of \cite[Lemma A.3]{RS}, 
\begin{flalign}
\wtTh(D) =\textup{tr arg}(I_{n} + \i D) & =\lim_{p \to \infty} \textup{tr arg}(I_n^{1/p^2} + \i D ) = \lim_{p \to \infty} \textup{tr arg}(I + \i I^p_{n}DI^p_{n}) \nonumber\\
&=\lim_{p \to \infty} \textup{tr tan}^{-1}(D_p)  \nonumber \\
&= \lim_{p \to \infty} \tilde \theta (D_p), \ \ \ \ D \in \textup{Sym}^2(\Bbb R^{n+1}) \setminus \mathcal S.\label{eq: trarg_limitformula}
\end{flalign}
It follows that there exists big enough $p$ such that ${A_k}_p,{B_k}_p \in F_{c+\varepsilon_k/2}$. As $F_{c+\varepsilon_k/2}$ is convex, this implies that ${C_k}_p \in F_{c+\varepsilon_k/2}$, i.e., $\tilde \theta({C_k}_p) \geq c +\varepsilon_k/2$ for large enough $p$. Using \eqref{eq: trarg_limitformula} again, it follows that ${C_k} \in \mathcal F_{c+\varepsilon_k/2}$, hence $C_k \in \mathcal F_{c}$. As $\mathcal F_c$ is closed, it follows $C \in \mathcal F_c$, implying that $\mathcal F_{c}$ is convex.

We argue now that $\mathcal F_c(I'' \times D'')$ is convex for any $I'' \subset I$ and $D'' \subset D$. Let 
$$
v,w \in \mathcal F_c(I'' \times D'').
$$
By \cite[Theorem 8.2]{HL} there exists $\{v_k\}_k,\{w_k\}_k \in C^0 \cap \mathcal F_c(I''_k \times D''_k)$ quasi-convex, such that $v_k \searrow v,w_k \searrow w$, and $I''_k\times D''_k$ exhausts $I''\times D''$. By quasi-convexity there exists $S \subset I'' \times D''$ of measure zero such that 
$$
\nabla^2 v_k(x),\nabla^2 w_k(x) \in \mathcal F_c, \q x \in I''_k \times D''_k \setminus S.
$$
Convexity of $\mathcal F_c$ now gives that $\nabla^2 (v_k(x)+ w_k(x))/2 \in \mathcal F_c, x \in I''_k \times D''_k \setminus S$. As $(v_k(x)+ w_k(x))/2$ is also quasi-convex,  then  $(v_k + w_k)/2 \in \mathcal F_c(I''_k \times D''_k)$
\cite[Corollary 7.5]{HL}. The fact that $(v_k + w_k)/2 \searrow (v + w)/2$ and that $I''_k\times D''_k$ is exhaustive implies that $(v+w)/2 \in \mathcal F_c(I'' \times D'')$.

We turn to the last statement of the Lemma. Let $\tilde I \subset I,\tilde D\subset D$ be open neighborhoods of $\overline {I'}, \overline {D'}$. It follows from \cite[Theorem 8.2]{HL} that there exists $\{u'_k\}_k \in C^0 \cap \mathcal F_c(\tilde I \times \tilde D)$ strictly decreasing to $u|_{\tilde I \times \tilde D}$. Actually, the potentials $u'_k$ are quasi-convex 
hence continuous. Approximate each $u'_k$ locally uniformly with 
$$
\{u'_{k,l}\}_{l\in\NN}
\subset  C^\infty \cap \mathcal F_c(I' \times D')
$$ 
using \cite[Lemma 10.7]{RS}
(this last result is applicable since, as proven above, $\mathcal F_c(I' \times D')$ is convex).
Since $u'_k<u'_{k-1}$, there exists $N(k)\in\NN$ such that 
$u'_k<u'_{k,N(k)}<u'_{k-1}$. 
Consequently, the sequence $\{u'_{k,N(k)}\}_{k\in\NN}\subset  C^\infty \cap \mathcal F_c(I' \times D')$ is  decreasing to $u|_{I' \times D'}$, as desired. 
\end{proof}

\section{A formula for solutions of the DSL}
\label{sec: applications}

Given a function $f=f(t,x)$ on $I\times D$ (that we consider
as a family of functions on $I$ parametrized by $D$), we let
\beq\label{LegTrans1Eq}
f^\star(\tau,x)=f^\star(\tau):=\inf_{t\in I}[f(t,x)-\tau].
\eeq
This is the {\it negative} of the usual partial Legendre transform solely in the $t$-variable. Despite this, we also refer to it sometimes as the partial Legendre transform, and we often omit the dependence
of the function on the $D$ variables in the notation. Conversely,
if $g=g(\tau,x)$ is a
function on $\RR\times D$ taking values in $[-\infty,\infty)$,
where $\RR$ is considered as the dual vector space to the copy of $\RR$ containing $I$, 
then let
\beq\label{LegTrans2Eq}
g^\star(t,z)=g^\star(t):=\sup_{\tau\in \RR}[g(\tau,x)+\tau t].
\eeq
Note that $f^{\star\star}=f$ if and only if $f$ is convex in $t$, lower semicontinuous and nowhere
equal to $-\infty$ (we do not allow the constant function $-\infty$ in this section) \cite[Theorem 12.2]{Rockafellar}.

Let $I=[0,1]$. Given a function $g$ on $\del(I\times D)$, we say a function $u\in C^2(I\times D)$ solves the
Dirichlet problem for the  degenerate special Lagrangian (DSL) equation of phase $\th\in(-\pi,\pi]$ if
\beq
\label{DSLMainEq}
\baeq
\Im \left( e^{-\i \th} \det (I_n + \i\nabla^2 u)\right) 
& = 0, 
\h{\ \ on \ \ } I\times D,
\cr 
\Re \left(e^{-\i \th} \det \left(I + \i\nabla_x^2 u\right)\right) 
& > 0, 
\h{\ \ on \ \ } I\times D,
\cr
u
&=g \h{\ \ on \ \ }\del(I\times D).
\eaeq
\eeq
Recall from \S\ref{DLSSubSec} that a function $u$ is said to be a weak solution of the Dirichlet problem for the  DSL if 
$u|_{\del(I\times D)}=g$ and $u\in \calF_c(I\times D)\cap (-\calF_c(I\times D))$, where $c \in (-(n+1)/2\pi,(n+1)/2\pi)$ and $c \equiv \theta \mod 2\pi$. 
Weak solutions to the Dirichlet problem for the  DSL exist by the following result 
\cite[Theorem 1.2]{RS}.
\bthm
\label{MainExistenceThm}
Let $D\subset \RR^n$ be a bounded strictly convex domain,
and let $g\in C^2\left(\del(I\times D)\right)$ be a consistent function
such that 
\begin{equation}\label{eq:plbc}
g|_{\{i\}\times D}
\in F_{c-\pi/2}(D),
\end{equation}
for $i\in\{0,1\}$, 
with $c\in [n\pi/2,(n+1)\pi/2)$.
There exists a unique solution $u 
\in C^0(\overline{I\times D})\cap C^{0,1}(I\times D)$ for the $\calF_c$-Dirichlet problem with boundary values $g$.  
\ethm

Given $v: D\ra\RR$ and $f:\del D\ra\RR$, define the 
$(F_a,v,f)$-envelope
$$
P(v;f)=\sup\{w\in F_a(D)\,:\, w\le v \h{\ on \ }  D,\; w|_{\del D}\le f\},$$
where $w|_{\del D}\le f$ means that $\limsup_{\xi \to x}w(\xi) \leq f(x)$ for all $x \in \del D$.
\begin{lemma} 
\label{uscLemma}
If $a \in [(n-1)\pi/2,n\pi/2)$ then $P(v,f)=\usc P(v;f)\in F_a(D)$
and moreover $P(v;f)\in C^0(D)$. Also if $v,f$ are continuous, then $P(v,f) \leq v$ and $P(v,f)|_{\del D} \leq f$, i.e., $P(v,f)$ is a ``candidate for itself''.
\end{lemma}
\begin{proof} By \cite[(6), p. 410]{HL} it follows that $\usc P(v;f)\in F_a(D)$. The fact that 
\begin{equation}
\begin{aligned}
\label{uscPvfEq}
P(v,f)=\usc P(v;f)
\end{aligned}
\end{equation}
 can be proved as follows.
First, $F_a \subset \mathcal P$ (and hence $F_a(D) \subset \mathcal P(D)$ 
\cite[(4.2), p. 409]{HL})
\cite[Lemma 10.4]{RS}, hence $P(v,f)$ is a supremum of convex functions, hence convex
\cite[Theorem 5.5]{Rockafellar}. Thus it is continuous if it is locally bounded. 
It is certainly bounded from above in terms of $v$ and 
$f$. As convex functions are automatically lsc, it is also bounded from below.
Thus, $P(v,f)=\usc P(v;f)$ and so in particular also \eqref{uscPvfEq} holds.

Now we focus on the last statement of the Lemma. Clearly, $P(v,f) \leq v$, by continuity of $v$. For the inequality at the boundary, notice that 
$$
P(v,f) 
\leq 
u(f):=\sup\{w\in  F_a(D)\,:\,  w|_{\del D}\le f\}.
$$ 
According to \cite[Theorem 6.2]{HL},
$u(f) \in  F_a(D)$ is the unique continuous (up to the boundary) 
solution of the Dirichlet problem associated to the 
subequation $F_a$ on $D$ 
with boundary value $f$ 
(since $D$ is bounded and strictly convex domain it also satisfies the boundary assumptions of op. cit., see, e.g., \cite[Remark 8.2]{RS}).
In sum,  $P(v,f)|_{\del D} \leq u(f)|_{\del D}= f$, as desired.
\end{proof}

\begin{remark} 
{\rm
\label{}
In the last step of the proof we could have equally well have used the fact that
$$
P(v,f) 
\leq 
\sup\{w\in \mathcal P(D)\,:\,  w|_{\del D}\le f\},
$$ 
since as already noted $F_a(D) \subset \mathcal P(D)$.
As is well known, 
the right hand side is the unique convex continuous (up to the boundary) 
solution of the Dirichlet problem associated to the homogeneous
real Monge--Amp\`ere equation on (the bounded and strictly convex domain) $D$ 
with boundary value $f$ \cite[Theorem 2.8]{RauchTaylor}. 
This also implies that $P(v,f)|_{\del D} \leq  f$.
Of course, the theorem of Harvey--Lawson is more general.
A small advantage of the proof given above is that it carries over verbatim to domains $D$ which are merely strictly
$\vec F_a$ and $\vec {\tilde F}_a$ convex, cf. \cite{RS,HL}.
} \end{remark}

We now state the main result of this section. It shows that the solution of the Dirichlet problem for the DSL can be expressed as the partial Legendre transform of a family 
of solutions of obstacle problems for the {\it non-degenerate } special Lagrangian equation. This is inspired by and stands in clear analogy to a result on the
homogeneous real/complex Monge--Amp\`ere equation \cite[Corollary 2.2, 
Proposition 2.3]{DR}.

\begin{theorem}
\label{geodThm}
Let $u$ be given by Theorem \ref{MainExistenceThm}.
Then for any $(t,x) \in I \times D$ we have\begin{equation}\label{main_identity}
\baeq
u(t,x)
&=
\bigg(
P_{c-\pi/2}
\Big(
\min\{g|_{\{0\}\times D}, g|_{\{1\}\times D}-\tau\};
\inf_{r\in[0,1]}
(g|_{(0,1)\times\del D}(r,\,\cdot\,)-r\tau)
\Big)
\bigg)
^\star(t,x).
\eaeq
\end{equation}
\end{theorem}
\begin{proof}
By Lemma \ref{cvxLemma}, the function $t \to u(t,x)$ is convex.
Thus, $u^{\star\star}=u$, and hence it suffices to show that
\begin{equation}\label{eq: idtoprove}
u^\star(\tau,x)=
\inf_{t\in I}[u(t,x)-\tau t]
=P_{c-\pi/2}
\Big(
\min\{g|_{\{0\}\times D}, g|_{\{1\}\times D}-\tau\};
\inf_{r\in[0,1]}
(g|_{(0,1)\times\del D}-r\tau) 
\Big).
\end{equation}
Throughout the rest of the proof we fix $\tau$.
Let us denote the upper envelope on the left hand side by 
$$
h(\tau,x)=h_\tau(x).
$$ 
As $u \in \mathcal F_c([0,1] \times D)$, also 
$u(t,x)-\tau t\in \mathcal F_c([0,1] \times D)$ \cite[(2), p. 410]{HL}. Thus, by
Theorem \ref{mainThm}  
$$v_\tau(x):=u^\star(\tau,x) \in F_{c-\pi/2}(D).
$$
Hence, by the Dirichlet conditions on $u$ guaranteed by Theorem \ref{MainExistenceThm},
$$
v_\tau \leq \min\{g|_{\{0\}\times D}, 
g|_{\{1\}\times D}-\tau\} \ \ \textup{ and } \ \ v_\tau|_{\partial D} 
\leq \inf_{r\in[0,1]}
(g|_{(0,1)\times\del D}-r\tau).
$$ 
This implies that $v_\tau$ is a candidate in the definition of $h_\tau$, i.e., 
$v_\tau \leq h_\tau$.

We turn to prove the other inequality in \eqref{eq: idtoprove}. Notice that
$$
\min\{g|_{\{0\}\times D}, g|_{\{1\}\times D}-\tau\}\in C^0(D)
$$
since $g|_{\{0\}\times D},\, g|_{\{1\}\times D}-\tau\in C^0(D)$
by assumption. Also, letting $g_y(r):=g(r,y)$ for $y\in\del D$, 
$$
f_y(\tau):=\inf_{r\in[0,1]}(g|_{(0,1)\times\del D}(r,y)-r\tau)
=g_y^\star(\tau).
$$
Since $(r,y)\mapsto g_r(y)$ is continuous and $[0,1]$ is compact, it follows that also
$y\mapsto f_y(\tau)$ is $C^0$ in $y$ 
(Indeed, let $\eps>0$. 
Choose $\delta>0$ so that $|g_y(r)-g_z(r)|<\eps$ for all $z$ satisfying 
$|z-y|<\delta$. Then $g_y^\star(\tau)\le g_z^\star(\tau)+\eps$
and similarly $g_z^\star(\tau)\le g_y^\star(\tau)+\eps$.)
Combining these facts, Lemma \ref{uscLemma} implies that 
\begin{equation}
\begin{aligned}
\label{htauEq}
h_\tau(x)\in F_{c-\pi/2}. 
\end{aligned}
\end{equation}
We claim that
\begin{equation}
\begin{aligned}
\label{wtxEq}
w(t,x):=h_\tau(x)\in\mathcal F_c([0,1] \times D),
\end{aligned}
\end{equation}
i.e., $w$ is a (constant in $t$)  
 subsolution to the DSL equation \eqref{DSLMainEq}. 
This follows immediately from   \eqref{eq:whTh} if $h$ is $C^2$
since then $\wtTh_w(t,x)=\pi/2+ \wtth_{h_\tau}(x)\ge c$ by
 \eqref{htauEq}; otherwise, since 
$h_\tau\in C^0(D)$ by  Lemma \ref{uscLemma}, we can approximate $h_\tau$
  locally uniformly by smooth $F_c$-potentials \cite[Lemma 10.7] {RS}. Then we can apply \cite[(5'), p. 410]{HL} to conclude \eqref{wtxEq}.

We can conclude the proof by the standard argument that 
a subsolution lies below a solution. More precisely, by Theorem \ref{MainExistenceThm} and \cite[(2), p. 410]{HL},
$$
u-t\tau\in -\widetilde{\mathcal F}_c([0,1] \times D),
$$
so $w-u+t\tau\in\SA([0,1] \times D)$
\cite[Theorem 6.5] {HL}.
Since 
$$
w(t,x)-(u(t,x)-t\tau)\le0\q \h{ on}\q  \del([0,1] \times D),
$$
it follows that
$w(t,x) \leq u(t,x)-t\tau$ \cite[Proposition 2.3] {HL}. We have shown that
$$
h_\tau(x) 
\leq v_\tau(x)= \inf_{t \in [0,1]}[u(t,x)-\tau t],
$$ 
giving the other direction of \eqref{eq: idtoprove}.
\end{proof}

\section*{Acknowledgments}
This work was supported by BSF grant 2012236, 
NSF grants DMS-1610202,1515703, and a Sloan Research Fellowship.
Part of this work took place at MSRI (supported by NSF grant DMS-1440140)
during the Spring 2016 semester. The authors are grateful to J. Solomon and A. Yuval for many helpful discussions.

\def\listing#1#2#3{{\sc #1}:\ {\it #2}, \ #3.}

\begin{small}

\vspace{.3 cm}
\noindent

{\sc University of Maryland}

{\tt tdarvas@math.umd.edu, yanir@umd.edu}
\end{small}

\end{document}